\documentclass[11pt,letterpaper]{amsart}

\usepackage{header}
\numberwithin{equation}{subsection}
\begin{document} 
\title{Spectral Multiplicity Bounds for Jacobi Operators on Star-Like Graphs}
\author{Netanel Levi}
\thanks{Netanel Levi, Department of Mathematics, University of California, Irvine. Supported by NSF DMS-2052899, DMS-2155211, and Simons 896624. Email: netanell@uci.edu}
\author{Tal Malinovitch}
\thanks{Tal Malinovitch, Department of Mathematics, Rice Unniversity. Email: tal.malinovitch@rice.edu}
\begin{abstract}
We study the spectral multiplicity of Jacobi operators on star-like graphs with $m$ branches. Recently, it was established that the multiplicity of the singular continuous spectrum is at most $m$. Building on these developments and using tools from the theory of generalized eigenfunction expansions, we improve this bound by showing that the singular continuous spectrum has multiplicity at most $m-1$. We also show that this bound is sharp, namely, we construct operators with purely singular continuous spectrum of multiplicity $m-1$.
\end{abstract}
\maketitle
\section{Intro}
In this work, we study the spectral multiplicity of Jacobi operators on star-like graphs. We will assume throughout that the graphs we discuss are infinite, though locally finite. Given a graph $\calG=\left(\calV,\calE\right)$, a Jacobi operator on $\calG$ is a linear operator $J:\ell^2\left(\calV\right)\to\ell^2\left(\calV\right)$ which is given by
\begin{equation}\label{op_eq}
    J\psi(v)=\sum_{w\sim v} a_{v,w} \psi(w)+b_v\psi(v).
\end{equation}
Here, $a:\calE\to\mathbb{R}$ and $b:\calV\to\mathbb{R}$ are both bounded functions. We will also assume that $a$ is positive. Under these assumptions, the operator $J$ is bounded and self-adjoint. We will focus on the spectral multiplicity of such operators.

The spectral multiplicity of self-adjoint operators has been the subject of many works; some examples are  \cite{gilbert1998subordinacy,halmos2017introduction,howland1986finite,kats1963spectral,levi2023subordinacy,liaw2020matrix,simon2004theorem,simonov2014spectral,simonov2024local}. A self-adjoint operator $T$ can be associated with a Borel measure $\mu$ and a multiplicity function $N:\sigma\left(T\right)\to\mathbb{N}\cup\left\{\infty\right\}$, where $\sigma(T)$ is the spectrum of $T$, and $N$ is defined $\mu$-almost everywhere. One reason for studying $\mu$ and $N$ is, for example, the fact that the pair $\left(\mu,N\right)$ determines a self-adjoint operator up to unitary equivalence. In the pure point spectrum, the multiplicity function $N\left(E\right)$ is precisely the dimension of the eigenspace corresponding with $E$. In the case of Jacobi operators on graphs, an important example is that of a Jacobi operator $J$ on  $\calG=\mathbb{N}$, where $x,y\in\mathbb{N}$ are neighbors if and only if $\left|x-y\right|=1$. In that case, it is not hard to see that the vector $\delta_1$ is cyclic for $J$ (see Section \ref{prelim_sec} for precise definitions). It is known that, in general, the multiplicity is bounded from above by the size of a cyclic set; this implies that for Jacobi operators on $\ell^2\left(\mathbb{N}\right)$, $N$ is equal to $1$ $\mu$-almost everywhere, see \cite{halmos2017introduction}.

Unlike the case of $\mathbb{N}$, Jacobi operators on $\ell^2\left(\mathbb{Z}\right)$ do not necessarily have a cyclic vector, and so this argument no longer works. In this case, while the absolutely continuous spectrum may have multiplicity $N=2$ (see, e.g., \cite[Chapter 10]{kato2013perturbation}), it was proved by Kac \cite{kats1963spectral} that this is not possible in the singular spectrum. Namely, he proved that if we denote the singular part of $\mu$ w.r.t.\ the Lebesgue measure by $\mu_s$, then for $\mu_s$-almost every $E\in\mathbb{R}$, $N\left(E\right)=1$. Since then, this result was proved in several different ways \cite{gilbert1998subordinacy,levi2023subordinacy,simon2004theorem} and generalized to many other settings \cite{howland1986finite,levi2023subordinacy,liaw2020matrix,simon2004theorem,simonov2014spectral,simonov2024local}.

In \cite{gilbert1998subordinacy,levi2023subordinacy}, a connection was established between spectral multiplicity and subordinacy theory, which relates local singularity of $\mu$ to the existence of certain solutions to an associated eigenvalue equation (see (\ref{ev_eq_half_line}) below). Roughly speaking, a solution to the eigenvalue equation is called subordinate if it grows slower (at infinity) than any other solution. In \cite{gilbert1998subordinacy}, this connection, along with the uniqueness of subordinate solutions, was used in order to determine simplicity of the singular spectrum for operators on the line. This idea was then generalized in \cite{levi2023subordinacy} to the case of star-like graphs, which are given by pasting a finite number of half-lines to a compact graph (see Section \ref{prelim_sec} for a precise definition). For these graphs, it was shown that for $\mu_s$-almost every $E\in\mathbb{R}$, $N\left(E\right)$ is bounded from above by the number of linearly independent subordinate solutions.

The results of \cite{levi2023subordinacy} can be rephrased in terms of generalized eigenfunction expansion theory, see \cite{berezanskiui1968expansions} and references therein. It is well-known that the spectral resolution $P$ of a self-adjoint operator can be realized as an integral of generalized projections, namely
\begin{align*}
    P\left(A\right)=\int\limits_\mathbb{R}P\left(E\right)\,d\mu\left(E\right)
\end{align*}
for every Borel set $A\subseteq\mathbb{R}$, where $P\left(E\right)$ acts on a dense subspace and projects vectors to generalized eigenfunctions. Then, the multiplicity function $N(E)$ is given by \mbox{$N\left(E\right)=\dim\left(\Ran\left(P\left(E\right)\right)\right)$}. With that in mind, the results of \cite{levi2023subordinacy} essentially say that every $u\in\Ran\left(P\left(E\right)\right)$ is a subordinate solution. Thus, indeed, $N\left(E\right)$ is bounded from above by the maximal amount of linearly independent subordinate solutions.

While these results provide a seemingly complete picture in the singular part of the spectrum, the bounds can be improved upon further restrictions. Indeed, in the singular continuous part of the spectrum, a solution to the eigenvalue equation is necessarily not in $\ell^2\left(\calG\right)$, and in particular must be supported on one of the half-lines. Using this intuition, it was shown in \cite{levi2023subordinacy} that for $\mu_{\text{sc}}$-almost every $E\in\mathbb{R}$, $N\left(E\right)\leq m$, where $m$ is the number of half-lines attached to the compact component. Our main result is a further improvement of this bound.

\begin{theorem}\label{main_thm}
    Let $\calG$ be a star-like graph with $m$ branches and $J$ be a Jacobi operator on it, then for $\mu_{\text{sc}}$-almost every $E\in\mathbb{R}$, $N\left(E\right)\leq m-1$. Furthermore, for every $m$ there is a star like graph $\calG_m$ with $m$ branches such that for $\mu_{\text{sc}}$-almost every $E\in\mathbb{R}$, we have that $N\left(E\right)= m-1$. 
\end{theorem}
The proof of the first part of Theorem \ref{main_thm} has two main ingredients. The first is an eigenfunction expansion formula for star-like graphs, which is a generalization of the results found in \cite{levi2024eigenfunction}. Namely, we formally show that for $\mu_s$-almost every $E\in\mathbb{R}$, every $f\in\Ran P\left(E\right)$ is a subordinate solution. Using this, we show that the existence of a non-trivial Borel set on which the multiplicity is equal to $m$ implies the existence of an infinite-dimensional $J$-invariant subspace, which consists of functions that are supported on a single half-line. Finally, we turn to show that such subspaces cannot exist, so the multiplicity must be less than or equal to $m-1$.

In a recent work \cite{simonov2024local}, the multiplicity of the singular spectrum is analyzed for operators constructed by gluing together a finite number of self-adjoint operators with simple spectrum. While there is a certain connection, the results of \cite{simonov2024local} are not directly applicable to our framework, nor do our results immediately extend to their setting. However, we believe our results may be generalized to a setting similar to that of \cite{simonov2024local}. We discuss this in Section \ref{PastingSection}.

The paper is structured as follows. In Section \ref{prelim_sec}, we present some notations and preliminary results that will be used in the proof of Theorem \ref{main_thm}. In Section \ref{proof_section}, we prove Theorem \ref{main_thm}. Finally, in Section \ref{discuss_section}, we provide an example where the bound in Theorem \ref{main_thm} is achieved, and further discussion.

\subsection*{Acknowledgments} We would like to thank Jonathan Breuer for suggesting the problem and
for useful discussions.
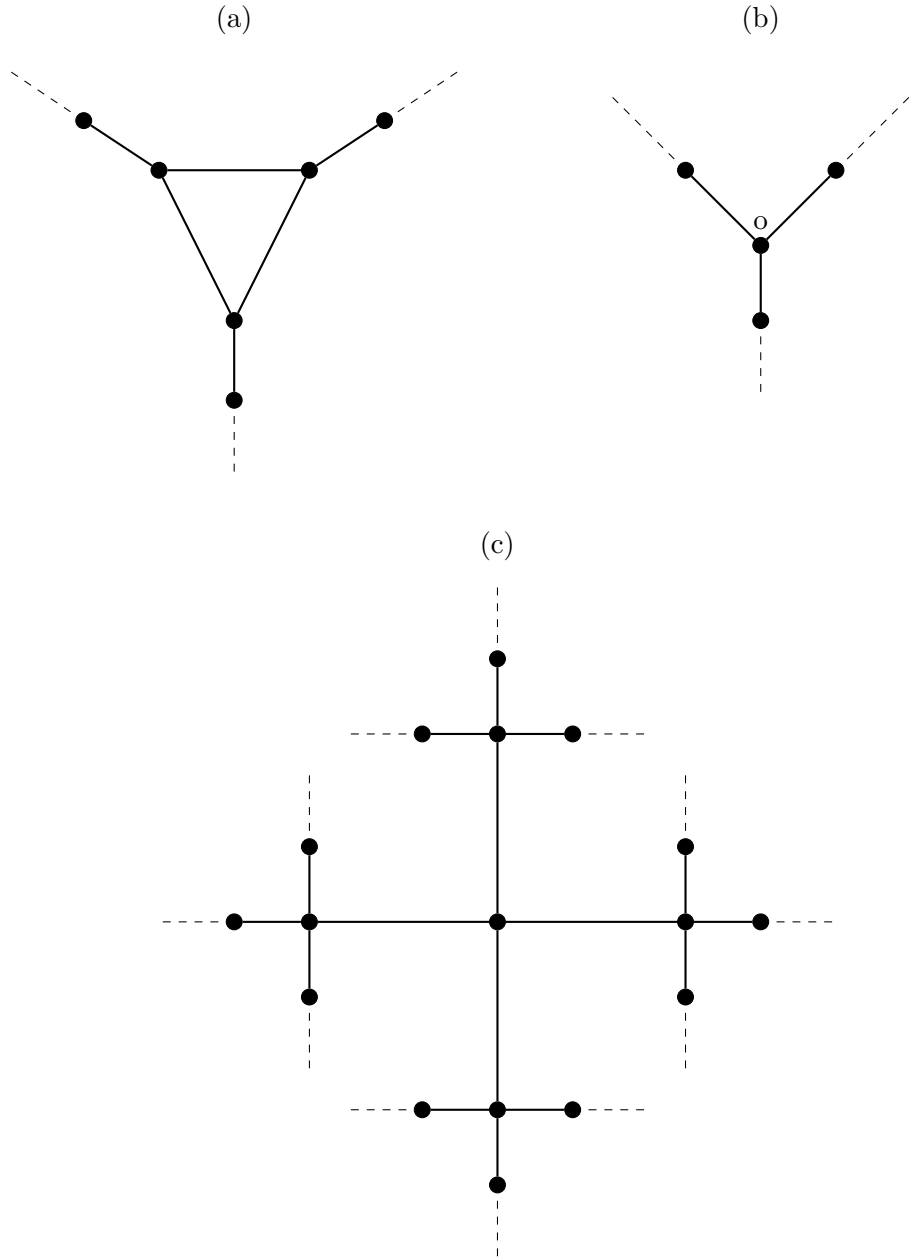
\begin{figure}
	\begin{center}
		\begin{tikzpicture}[scale=1]
			\vertex(t1) at (-6,1) {};
			\vertex(t2) at (-4,1) {};
			\vertex(t3) at (-5,-1) {};
			\vertex(t21) at (-3,1.66) {};
			\vertex(t11) at (-7,1.66) {};
			\vertex(t31) at (-5,-2.06) {};
			\Edge (t1)(t2);
			\Edge (t2)(t3);
			\Edge(t3)(t1);
			\Edge(t2)(t21);
			\Edge(t1)(t11);
			\Edge(t3)(t31);
			\draw (-5,-2.06) -- (-5,-3.06)[dashed];
			\draw (-3,1.66) -- (-2,2.33)[dashed];
			\draw(-7,1.66) -- (-8,2.33)[dashed];
			\node at (-5,3) {(a)};

			\vertex(so) at (2,0) {};
			\node at (2,0.3) {o};
			\vertex(s1) at (3,1) {};
			\vertex(s2) at (2,-1) {};
			\vertex(s3) at (1,1) {};
			\Edge (so)(s1);
			\Edge (so)(s2);
			\Edge (so)(s3);
			\draw (3,1) -- (4,2)[dashed];
			\draw (2,-1) -- (2,-2)[dashed];
			\draw (1,1) -- (0,2)[dashed];
			\node at (2,3) {(b)};
			
	\vertex(ro) at (-1.5,-9) {};
			\vertex(r1) at (-1.5,-6.5) {};
			\vertex(r11) at (-1.5,-5.5) {};
			\vertex(r12) at (-0.5,-6.5) {};
			\vertex(r13) at (-2.5,-6.5) {};
			\vertex(r2) at (1,-9) {};
			\vertex(r21) at (1,-8) {};
			\vertex(r22) at (2,-9) {};
			\vertex(r23) at (1,-10) {};
			\vertex(r3) at (-1.5,-11.5) {};
			\vertex(r31) at (-0.5,-11.5) {};
			\vertex(r32) at (-1.5,-12.5) {};
			\vertex(r33) at (-2.5,-11.5) {};
			\vertex(r4) at (-4,-9) {};
			\vertex(r41) at (-4,-10) {};
			\vertex(r42) at (-5,-9) {};
			\vertex(r43) at (-4,-8) {};
			\Edge (ro)(r1);
			\Edge(r1)(r11);
			\Edge(r1)(r12);
			\Edge(r1)(r13);
			\Edge(ro)(r2);
			\Edge(r2)(r21);
			\Edge(r2)(r22);
			\Edge(r2)(r23);
			\Edge(ro)(r3);
			\Edge(r3)(r31);
			\Edge(r3)(r32);
			\Edge(r3)(r33);
			\Edge(ro)(r4);
			\Edge(r4)(r41);
			\Edge(r4)(r42);
			\Edge(r4)(r43);
			\draw (-1.5,-5.5) -- (-1.5,-4.5)[dashed];
			\draw (-0.5,-6.5) -- (0.5,-6.5) [dashed];
			\draw (-2.5,-6.5) -- (-3.5,-6.5)[dashed];
			\draw (1,-8) -- (1,-7)[dashed];
			\draw (2,-9) -- (3,-9)[dashed];
			\draw (1,-10) -- (1,-11)[dashed];
			\draw (-0.5,-11.5) -- (0.5,-11.5)[dashed];
			\draw (-1.5,-12.5) -- (-1.5,-13.5)[dashed];
			\draw (-2.5,-11.5) -- (-3.5,-11.5)[dashed];
			\draw (-4,-10) -- (-4,-11)[dashed];
			\draw (-5,-9) -- (-6,-9)[dashed];
			\draw (-4,-8) -- (-4,-7)[dashed];
			\node at (-1.5,-4) {(c)};
		\end{tikzpicture}
	\end{center}
	\captionof{figure}{Three examples of star-like graphs. The dashed lines represent copies of $\mathbb{N}$. (a) is a star-like graph with $3$ branches,  where the compact component can be taken to be any finite subgraph that contains the inner triangle.
    (b) is also called a star graph with $3$ branches. The compact component can be taken to be any finite subgraph that contains the vertex $o$ and two of its neighbors.  (c) is a star-like graph with $12$ branches. Graph (c) is also called the trimming of a $4$-regular tree. In the language of Section \ref{branching_section}, the branching sequence for this graph is $\left(3,3,1,1,1,1,\ldots\right)$. This figure is taken from \cite{levi2023subordinacy}.} 
	\label{sg_fig}
\end{figure}
\section{Preliminaries}\label{prelim_sec}
\subsection{Notation and setting}
Throughout this paper, we will discuss Jacobi operators acting on a graph. We will denote a graph by $\calG=(\calV,\calE)$, where $\calV$ is the collection of vertices, and $\calE\subset \calV\times \calV$ is the collection of edges. Since we will be working with functions defined on the vertices of the graph, we distinguish between those in $\ell^2(\calG)$ and more general functions on $\calV$. To that end, we reserve the letters $\varphi, \psi, \phi$ (and other letters from the Greek alphabet) for functions in $\ell^2(\calG)$, the letters $f,g,h$ for arbitrary functions $\calV \to \mathbb{C}$ (not necessarily in $\ell^2(\calG)$), and the letters $u, v, w$ to denote vertices.

In the following, we collect some standard definitions:
\begin{definition}
\begin{itemize}
    \item Given $u,v\in \calV$, we will write $u\sim v$ if $(u,v)\in \calE$.
    \item A path between $u$ and $v$ is a finite sequence of vertices $\left(x_0=u,x_1,\ldots,x_n=v\right)$ such that for every $1\leq i\leq n$, $x_{i-1}\sim x_i$. The number of vertices in the path is called its length.
    \item  Given $u,v\in\calV$, we denote by $\dist\left(u,v\right)$ the minimal length of a path connecting $u$ and $v$.
    \item For any $v\in \calV$ and $n\in\mathbb{N}$, we denote $B_n\left(v\right)=\left\{u\in\calV:\dist{\left(u,v\right)}\leq n\right\}$, and the boundary  $\partial B_n\left(v\right)=\left\{u\in\calV:\dist{\left(u,v\right)}= n\right\}$. Note that for any $v\in \calV$ we have that \mbox{$\partial B_1(v)=\{u\mid u\sim v\}$}.
    \item Given a subset of vertices $\calA\subset \calV$ we define, for any $u \in \calV$,  $\dist\left(u,\calA\right)=\inf\limits_{v\in \calA} \dist(u,v)$.
\end{itemize}
\end{definition}
We will focus on star-like graphs:
\begin{definition}
    A graph $\calG$ is called a \textit{star-like graph with $m$ branches} if
    \begin{align*}
    \calV=\calK\cup\bigcup_{i=1}^mL_i
    \end{align*}
    Where $\calK$ is a finite set, and for each $L_i $, there exists $\varphi_i :\bbN\cup\left\{0\right\}\rightarrow L_i$ with the following properties:
    \begin{enumerate}
        \item  For any $j\geq  1$ we have that  $B_1(\varphi_i(j))=\{\varphi_i(j-1),\varphi_i(j+1)\}$.
        \item We have that $B_1(\varphi_i(0))\subset \calK\cup \{\varphi_i(1)\}$.
        \item $\calK \cap L_i=\{\varphi_i(0)\} $
    \end{enumerate}
\end{definition}

\begin{remark}
    \begin{enumerate}
        \item Intuitively speaking, a star-like graph with $m$ branches is given by pasting $m$ half-lines (i.e.\ copies of $\mathbb{N}$) to a compact graph. See Figure \ref{sg_fig} for some examples.
        \item As discussed in \cite{levi2023subordinacy}, the choice of $\calK$ is not unique. However, our results will not depend on this choice, so we fix $\calK$ and refer to it as \textit{the} compact component of $\calG$.
        \item The case where $\calK=\{o\}\cup \bigcup_{i=1}^m \{\varphi_i(0)\}$, and $B_1(\varphi_i(0))=\{o,\varphi_i(1)\}$ is called star graph. 
    \end{enumerate}
\end{remark}

Let $f:\calV\to\mathbb{N}$ be given by $f\left(v\right)=\dist\left(v,\calK\right)+1$. We will consider the spaces $\calH_\pm$ and $\calH$ given by
\begin{center}
    $\calH=\ell^2\left(\calV\right)$,\\$\text{}$\\
    $\calH_\pm=\ell^2\left(\calV;f^{\pm 1}\right)\coloneqq\left\{\psi:\calV\to\mathbb{C}:\sum\limits_{v\in V}\left|\psi\left(v\right)\right|^2 f^{\pm 1}\left(v\right)<\infty\right\}$.
\end{center}
\begin{remark}
    Given $g:\calV\to\mathbb{R}_{>0}$,  the space $\ell^2\left(\calV;g\right)$ is called a \textit{weighted $\ell^2$ space}, and it is a Hilbert space with respect to the inner product $\langle\psi,\varphi\rangle_g=\sum\limits_{v\in \calV}\overline{\psi\left(v\right)}\varphi\left(v\right)g\left(v\right)$.
\end{remark}
\begin{remark}
    We would often abuse notation and write $\ell^2(\calG)$, where we in fact mean $\ell^2(\calV)$. 
\end{remark}
A Jacobi operator on a star-like graph $\calG$ with $m$ branches is a linear operator $J:\ell^2(\calG)\rightarrow\ell^2(\calG)$, which is given by
\begin{align*}
    \left(J\psi\right)(v)=\sum_{w\in \partial B_1(v)} a_{v,w} \psi(w)+b_v\psi(v)
\end{align*}
for every $v\in \calV$. We assume that $a\in\ell^\infty\left(\calE,\mathbb{R}_{>0}\right)$ and $b\in \ell^\infty(\calG,\mathbb{R})$. In that case, $J$ is bounded and self-adjoint. The function $b$ is also called the potential function.\par
\begin{remark}
    \begin{enumerate}
        \item The assumption that $a>0$ is only needed to ensure that the sums of products of the weights along the shortest paths between certain vertices will not be zero. 
        \item We believe that our results should also hold in the case where $a,b$ are unbounded as long as the resulting operator is essentially self-adjoint. However, since some of the results we rely on in our proofs are only proved for bounded operators, we restrict ourselves to that case.
    \end{enumerate}
    
\end{remark}
From now on, we will just say that we have a Jacobi operator $J$ on a star-like graph $\calG$ with $m$ branches, where $L_i,\calK,\varphi_i$ are all implied in this statement. \par 
We will consider spectral measures of vectors w.r.t.\ $J$. To that end, given $\varphi,\psi\in\ell^2\left(\calG\right)$, the joint spectral measure of $\varphi$ and $\psi$ w.r.t.\ $J$ is given by (see, for example, \cite{RSvol1})
\begin{center}
    $\mu_{\varphi,\psi}\left(A\right)=\langle\mathbbm{1}_A\left(J\right)\varphi,\psi\rangle,\,\,\,\,\,\,\,A\in\text{Borel}\left(\mathbb{R}\right)$.
\end{center}
If $\varphi=\psi$ we will denote $\mu_\varphi\coloneqq\mu_{\varphi,\varphi}$.\par
Given $v\in\calV$, we define $\delta_v:\calV\to\mathbb{C}$ by $\delta_v\left(u\right)=\begin{cases}
    1 & u=v\\ 0 & \text{otherwise}
\end{cases}$. It was proved in \cite{levi2023subordinacy} that the set $\calC=\left\{\delta_v:v\in\calK\right\}$ is cyclic for $J$, namely
\begin{equation}\label{cyclicity_eq}
    \ell^2\left(\calG\right)=\overline{\vspan\left\{J^k\delta_v:k\in\mathbb{N}\cup\left\{0\right\},v\in\calK\right\}}.
\end{equation}
It is well-known that in that case, for every $\varphi,\psi\in\ell^2\left(\calG\right)$, $\mu_{\varphi,\psi}\ll\mu$, where $\mu=\sum\limits_{\phi\in\calC}\mu_\phi$. For the rest of this work, we fix $\mu=\sum\limits_{v\in\calK}\mu_{\delta_v}$. Given $u,v\in\calV$, we will use the shorthand $\mu_{uv}\coloneqq\mu_{\delta_u,\delta_v}$ and $\mu_v\coloneqq\mu_{\delta_v,\delta_v}$. Given any finite Borel measure $\rho$, we will denote by $\rho_s$ the part of $\rho$ which is singular w.r.t.\ the Lebesgue measure. We will also denote by $\rho_{sc}$ the part of $\rho$ which is singular continuous, namely singular w.r.t.\ the Lebesgue measure and has no atoms.
\subsection{Generalized eigenfunction expansion and direct integrals}
In this section, we state some well-known spectral theoretic results in the theory of generalized eigenfunction expansion and direct integrals as presented in \cite{berezanskiui1968expansions,berezansky1996functional}. In order to avoid technicalities, we phrase the results in terms of Jacobi operators on star-like graphs rather than in their full generality.

Since the set $\calC$ defined above is cyclic, for every $u,v\in\calV$, we have $\mu_{uv}\ll\mu$. Combining this with the fact that $\calV$ is countable, we can find a full measure set $A\in\text{Borel}\left(\mathbb{R}\right)$ such that on $A$, all of the Radon-Nikodym derivatives $\frac{d\mu_{uv}}{d\mu}$ are well-defined. Thus, the operator $E\to P\left(E\right)$, where
\begin{equation}\label{P_def_eq}
    \left(P\left(E\right)\right)_{uv}=\frac{d\mu_{uv}}{d\mu}\left(E\right)
\end{equation}
is defined $\mu$-almost everywhere. We will use the following result, which is essentially proved in \cite{berezanskiui1968expansions}:
\begin{theorem}\label{ef_exp_thm}
    For $\mu$-almost every $E\in\mathbb{R}$, $P\left(E\right)$ is a Hilbert-Schmidt operator from $\calH_+$ to $\calH_-$. Furthermore, for every Borel set $A\subseteq\mathbb{R}$ and every $\psi\in\calH_+$, we have
    \begin{align*}
        P\left(A\right)\psi=\int\limits_A P\left(E\right)\psi\,d\mu\left(E\right)
    \end{align*}
    In addition, there exists a choice, for $\mu$-almost every $E\in\mathbb{R}$, of a number $N(E)$ and linearly independent vectors $f_1^E,\ldots,f_{N\left(E\right)}^E\in\Ran P\left(E\right)$ such that for $\mu$-almost every $E\in\mathbb{R}$, for every $\psi\in\calH_+$,
    \begin{equation}\label{eq_ef_exp_inner_prod}
        P\left(E\right)\psi=\sum\limits_{i=1}^{N\left(E\right)}\langle\psi,f_i^E\rangle f_i^E.
    \end{equation}
\end{theorem}
\begin{remark}
    Note that by the definition of $\calH_\pm$, elements in $\calH_+$ are also in $\ell^2\left(\calG\right)$ while this is not necessarily true for elements in $\calH_-$. We note that the dual of $\calH_+$ is naturally $\calH_-$ with the standard inner product, thus, the inner product inside (\ref{eq_ef_exp_inner_prod}) is well defined, and in particular, a part of Theorem \ref{ef_exp_thm} is the assertion that the sums of the inner product $\sum\limits_{v\in\calV}\overline{\psi\left(v\right)}f_i^E\left(v\right)$ converge $\mu$-almost everywhere.
\end{remark}
Given $E\in\mathbb{R}$, the elements $\left(f_i^E\right)_{i=1}^{N\left(E\right)}$ are called the \textit{generalized eigenfunctions} corresponding with $E$. The operator
\begin{center}
    $\calH_+\ni\varphi\to\hat{\varphi}\left(E\right)\coloneqq\left(\langle \varphi,f_1^E\rangle,\ldots,\langle\varphi,f_{N\left(E\right)}^E\rangle\right)\in\mathbb{C}^{N\left(E\right)}$
\end{center}
is called the \textit{Fourier transform} which corresponds with $J$. We call $N(E)$ the multiplicity at energy $E$. We now turn to the direct integral formulation of the spectral theorem as described in \cite[Chapter 15]{berezansky1996functional}.
\begin{theorem}\label{spect_thm_direct_int}
    Suppose there exists $N\in\mathbb{N}$ such that for $\mu$-almost every $E\in\mathbb{R}$, $N\left(E\right)\leq N$. Then, the Fourier transform can be extended to $\calH$. Furthermore, the operator $U:\calH \rightarrow \bbC^{N(E)}$, given by 
    \begin{align*}
        U\varphi=\hat{\varphi}
    \end{align*}
    forms a unitary equivalence between the action of $J$ on $\calH$ and the action $f\left(E\right)\to E\cdot f\left(E\right)$ on the direct integral $\int\limits^\oplus_\mathbb{R}\mathbb{C}^{N\left(E\right)}\,d\mu\left(E\right)$.

\end{theorem}
\begin{remark}
    In \cite{berezansky1996functional}, this is proved without requiring $N\left(E\right)$ to be bounded. However, since this is the case in our setting, we added this assumption to avoid some additional technicalities.
\end{remark}
We will need the following
\begin{lemma}\label{gen_ef_vanishing_lemma}
    Let $\psi\in\calH$. Suppose that for some $w\in \calV$, for $\mu$-almost every $E\in\mathbb{R}$, we have that 
    \begin{align*}
        \sum\limits_{i=1}^{N\left(E\right)}\langle\psi,f_i^E\rangle f_i^E\left(w\right)=0
    \end{align*}
    Then $\psi\left(w\right)=0$.
\end{lemma}
\begin{proof}
    Recall that the Fourier transform $\psi\to\hat{\psi}$ is unitary. In addition, we will use the fact that for $\mu$-almost every $E\in\mathbb{R}$, for every $1\leq i\leq N\left(E\right)$ and for every $w\in\calV$, $f_i^E\left(w\right)\in\mathbb{R}$. This follows, for example, from (\ref{P_def_eq}) along with the fact that $\mu_{uv}$ and $\mu$ are finite real measures. With that in mind, we have
    \begin{center}
        $\psi\left(w\right)=\langle\delta_w,\psi\rangle=\langle\hat{\delta_w},\hat{\psi}\rangle=\int\limits_\mathbb{R}\sum\limits_{k=1}^{N\left(E\right)}\langle\psi,f_i^E\rangle\overline{\langle\delta_w,f_i^E\rangle}\, d\mu\left(E\right)=\int\limits_\mathbb{R}\sum\limits_{k=1}^{N\left(E\right)}\langle\psi,f_i^E\rangle f_i^E\left(w\right)\, d\mu\left(E\right)=0$,
    \end{center}
    as required.
\end{proof}
\subsection{Subordinacy theory}
Let us briefly consider a Jacobi operator on $\ell^2\left(\mathbb{N}\right)$. To avoid confusion, we will denote it by $J^0$.  In that case, the vector $\delta_1$ is cyclic for $J^0$ and so $\mu^0=\mu_{\delta_1}^0$. Subordinacy theory relates local singularity properties of $\mu^0$ to certain asymptotic properties of solutions to the eigenvalue equation. For the following, we may consider 
\begin{align*}
    \calF=\{f:\bbN\cup \{0\}\rightarrow\bbR\}
\end{align*}
To that end, given $E\in\mathbb{R}$, we consider $g\in \calF$ which satisfies
\begin{equation}\label{ev_eq_half_line}
    a_{n,n-1}g\left(n-1\right)+a_{n+1,n}g\left(n\right)+b_ng\left(n\right)=Eg\left(n\right),\,\,\,\,n\in\mathbb{N},
\end{equation}
where $a_{1,0}=0$. Given $\theta\in\left[0,\pi\right)$, we denote by $g_\theta^E$ the unique solution to (\ref{ev_eq_half_line}) which also satisfies
\begin{equation}\label{bcon_eq}
    \left(g_\theta^E\left(0\right),g_\theta^E\left(1\right)\right)=\left(-\sin\theta,\cos\theta\right).
\end{equation}
Given $L\geq 1$ and $g:\mathbb{N}\to\mathbb{R}$, we denote
\begin{center}
    $\|g\|_L=\left(\sum\limits_{k=1}^{\lfloor L\rfloor}\left|g\left(k\right)\right|^2+\left(L-\lfloor L\rfloor\right)\left|g\left(\lfloor L\rfloor +1\right)\right|^2\right)^{\frac{1}{2}}$.
\end{center}
\begin{definition}
    Given $E\in\mathbb{R}$ and $\theta\in\left[0,\pi\right)$, $g_\theta^E$ is called \textit{subordinate} if for every $\theta\neq\eta\in\left[0,\pi\right)$,
    \begin{align}\label{subordinateEq}
        \underset{L\to\infty}{\lim}\frac{\left\|g_\theta^E\right\|_L}{\left\|g_\eta^E\right\|_L}=0
    \end{align}
\end{definition}

\begin{remark}
    \begin{enumerate}
        \item We note that one may consider the definition of subordinate solutions using the language of weak solutions, which is more commonly used in the realm of operators with an underlying continuous space. We may consider the dual of $\calF$, the space of compactly supported functions:
    \begin{align*}
        &F=\{f:\bbN\cup \{0\}\rightarrow \bbR\mid \#\{v\in \supp f\}<\infty \}.
    \end{align*}
     Then we call $g\in \calF $ a weak solution to the equation $Jg=Eg$ if for any $\psi \in F$ we have that 
    \begin{align*}
        \braket{J\psi, g}=E\braket{\psi,g}
    \end{align*}
    In this language, we see that a subordinate solution is a weak solution that is an asymptotic minimizer with respect to $\|\cdot \|_L$, in the sense of (\ref{subordinateEq}). Though there are some advantages to this language, it will not be used here in order to avoid confusion.
    \item Let $E\in\mathbb{R}$ and let $g$ be a solution to (\ref{ev_eq_half_line}) which does not necessarily satisfy (\ref{bcon_eq}). Denote $\widetilde{g}=\left(g\left(0\right),g\left(1\right)\right)$. We will say that $g$ is subordinate if and only if $\frac{g}{\|\widetilde{g}\|}$ is subordinate. 
    \end{enumerate} 
\end{remark}
The following can be easily proven using standard rank-one perturbation arguments and \cite[Theorem 1.2]{jitomirskaya1999power} (see also \cite[Chapter 2]{damanik2022one}).
\begin{lemma}\label{subordinacy_uniqueness}
    Let $J^0$ be a Jacobi operator on $\ell^2\left(\mathbb{N}\right)$. Let $E\in\mathbb{R}$ and suppose that $g_1,g_2$ are both subordinate solutions to (\ref{ev_eq_half_line}). Then there exists (a unique) $\lambda\in\mathbb{R}$ such that $g_1=\lambda g_2$. In other words, a subordinate solution to (\ref{ev_eq_half_line}) is unique up to scalar multiplication.
\end{lemma}
Let us now return to the context of a Jacobi operator on a star-like graph $\calG$, with $m$ branches. In this case, we will consider functions $g:\calV\to\mathbb{R}$, which satisfy
\begin{equation}\label{ev_eq_sg}
    \sum_{w\in \partial B_1\left(v\right)} a_{v,w} g(w)+b_vg(v)=Eg\left(v\right),\,\,\,\,v\in V.
\end{equation}
\begin{definition}
    A solution $g$ to (\ref{ev_eq_sg}) is called subordinate if
    \begin{enumerate}
        \item $g\not \equiv 0$.
        \item The restriction of $g$ to each half-line is either subordinate or identically equal to $0$.
    \end{enumerate}
\end{definition}
An immediate corollary of Lemma \ref{subordinacy_uniqueness} is the following.
\begin{corollary}\label{UniqueScalar}
    If $g$ is a subordinate solution with energy $E$ on a general star-like graph with $m$ branches, then for every $1\leq i\leq m$, there exists a unique scalar $\lambda_i \in \bbR$ such that 
    \begin{align*}
        \forall n\in \bbN, g(\varphi_i(n))=\lambda_i g_i(n)
    \end{align*}
    where $g_i$ is the unique subordinate solution of $J_i^0$ - the restriction of $J$ to $L_i$ - which satisfies the boundary condition (\ref{bcon_eq})  for some $\theta \in [0,\pi)$.
\end{corollary}
The following was proved in \cite{levi2023subordinacy}.
\begin{lemma}
    Given $E\in\mathbb{R}$, denote by $\calS\left(E\right)$ the set of all subordinate solutions corresponding with $E$. Then, $\calS\left(E\right)$ is a finite-dimensional vector space. Furthermore, for $\mu_s$-almost every $E\in\mathbb{R}$ we have
    \begin{equation}\label{dimension_sub_sol_multiplicity}
        N\left(E\right)\leq\dim\calS\left(E\right).
    \end{equation}
\end{lemma}
\subsection{The Borel transform}
Let $\rho$ be a finite real Borel measure. The Borel transform of $\rho$ is given by
\begin{center}
    $\calM_\rho\left(z\right)=\int\limits_\mathbb{R}\frac{d\rho\left(x\right)}{x-z}$,
\end{center}
where $z\in\mathbb{C}_+\coloneqq\left\{z\in\mathbb{C}:\text{Im} z>0\right\}$. It is an analytic function which maps $\mathbb{C}_+$ to itself. The boundary behavior of $\calM_\rho$ is strongly connected to various properties of the measure $\rho$. In particular, we will use the following, which was proved in \cite{poltoratskii1994boundary}.
\begin{lemma}\label{pol_thm}
    Let $\rho$ and $\sigma$ be finite positive measures. Assume that $\rho\ll\sigma$. Then for $\sigma_s$-almost every $E\in\mathbb{R}$,
    \begin{center}
        $\underset{\epsilon\to0}{\lim}\,\frac{\calM_\rho\left(E+i\epsilon\right)}{\calM_\sigma\left(E+i\epsilon\right)}=\frac{d\rho}{d\sigma}\left(E\right)$.
    \end{center}
\end{lemma}
We note that in the case where $\rho=\mu_{u,v}$, for some $u,v \in \calV$, we have that 
\begin{align*}
    \calM_\rho(E+i\epsilon)=\braket{R(E+i\epsilon)\delta_v,\delta_u}
\end{align*}
where $R(E+i\epsilon)=(J-(E+i\epsilon))^{-1}$ is the resolvent, which is bounded operator on $\ell^2(\calG)$ for any $\epsilon>0$. \par
We will also need the following two lemmas.
\begin{lemma}\label{m_func_bdd_by_epsilon}
    Let $\mu$ be a finite Borel measure and let $\calM$ be its Borel transform. Then for $\mu$-almost every $E\in\mathbb{R}$ there exists a constant $C>0$ such that for sufficiently small $\epsilon>0$, $\epsilon\cdot\left|\calM\left(E+i\epsilon\right)\right|<C$.
\end{lemma}
\begin{proof}
    Note that for $z=E+i\epsilon$, we have
    \begin{center}
        $\text{Im}\,\calM\left(E+i\epsilon\right)=\int\limits_\mathbb{R}\frac{\epsilon}{\epsilon^2+\left(x-E\right)^2}\,d\mu\left(x\right)$,\\$\text{}$\\
        $\text{Re}\,\calM\left(E+i\epsilon\right)=\int\limits_\mathbb{R}\frac{x-E}{\left(x-E\right)^2+\epsilon^2}\,d\mu\left(x\right)$.
    \end{center}
    In both cases, multiplying the integrand by $\epsilon$ yields an expression which is bounded from above by 1 and so the integral will be bounded from above by $2\cdot\left|\mu\right|$, where $\left|\mu\right|$ is the total variation of $\mu$.
\end{proof}

\begin{lemma}\emph{\cite[Remark 3.2]{levi2023subordinacy}}\label{subordinate_solution_initial_values}
    For $\mu_s$-almost every $E\in\mathbb{R}$, for every solution $g$ of (\ref{ev_eq_sg}), if there exists $v\in\calV$ such that for every $u\in\calK$ we have $g\left(u\right)=\frac{d\mu_{uv}}{d\mu}\left(E\right)$, then $g$ is subordinate.
\end{lemma}
We can now prove a statement about generalized eigenfunctions, which is crucial for the proof of Theorem \ref{main_thm}. We have
\begin{lemma}\label{gen_ef_sub_sol}
    For $\mu_s$-almost every $E\in\mathbb{R}$, every generalized eigenfunction corresponding with $E$ is a subordinate solution of (\ref{ev_eq_sg}).
\end{lemma}
\begin{proof}
    It suffices to show the lemma for $\mu_s$-almost every $E\in\mathbb{R}$, for every $v\in\calV$, $P\left(E\right)\delta_v$, as defined in Theorem \ref{ef_exp_thm}, is a subordinate solution of (\ref{ev_eq_sg}), as they span the generalized eigenfunctions. Fix $v\in \calV$ and denote $g=P\left(E\right)\delta_v$. Note that by definition we have  \mbox{$g\left(u\right)=\frac{d\mu_{uv}}{d\mu}\left(E\right)$}. We will first show that this function is a solution of (\ref{ev_eq_sg}). Let us denote the Borel transforms of $\mu_{uv}$ and of $\mu$ by $\calM_{uv}$ and by $\calM$, respectively. By Lemma \ref{pol_thm}, we may write
    \begin{center}
        $g\left(u\right)=\underset{\epsilon\to 0}{\lim},\frac{\calM_{uv}\left(E+i\epsilon\right)}{\calM\left(E+i\epsilon\right)}$.
    \end{center}
    Recall that for every $z\in\mathbb{C}_+$,
    \begin{align*}
        \calM_{uv}\left(z\right)=\int\limits_\mathbb{R}\frac{d\mu_{uv}\left(x\right)}{x-z}=\langle\left(J-z\right)^{-1}\delta_u,\delta_v\rangle=\langle\delta_u,\left(J-\overline{z}\right)^{-1}\delta_v\rangle
    \end{align*}
    Let us denote $R_\epsilon=\left(J-E+i\epsilon\right)^{-1}\delta_v$. Then we obtain
    \begin{align}\label{eq_RND_sol}
         \sum\limits_{w\in N\left(u\right)}a_{w,u}g\left(w\right)+b_ug\left(u\right)&=\underset{\epsilon\to0}{\lim}\frac{1}{\calM\left(E+i\epsilon\right)}\left[\sum\limits_{w\in N\left(u\right)}a_{w,u}R_\epsilon\left(w\right)+b_uR_\epsilon\left(u\right)\right]\\
         &=\underset{\epsilon\to0}{\lim}\,\frac{1}{\calM\left(E+i\epsilon\right)}\left(JR_\epsilon\left(u\right)\right).\nonumber
    \end{align}
    By the definition of $R_\epsilon$, denoting $z=E+i\epsilon$, we have
    \begin{center}
        $JR_\epsilon=\left(J-\overline{z}+\overline{z}\right)R_\epsilon=\delta_v+\overline{z}R_\epsilon$.
    \end{center}
    Plugging this in (\ref{eq_RND_sol}), we obtain
    \begin{center}
        $\sum\limits_{w\in N\left(u\right)}a_{w,u}g\left(w\right)+b_ug\left(u\right)=\underset{\epsilon\to0}{\lim}\,\frac{1}{\calM\left(E+i\epsilon\right)}\left(\delta_v\left(u\right)+ER_\epsilon\left(u\right)-\epsilon R_\epsilon\left(u\right)\right)=Eg\left(u\right)-\underset{\epsilon\to 0}{\lim}\,\frac{\epsilon R_\epsilon\left(u\right)}{\calM\left(E+i\epsilon\right)}$.
    \end{center}
    Finally, we have 
    \begin{align*}
        \epsilon R_\epsilon\left(u\right)=\epsilon\langle\delta_u,\left(J-\overline{z}\right)^{-1}\delta_v\rangle=\langle\left(J-z\right)^{-1}\delta_u,\delta_v\rangle=\calM_{uv}\left(z\right).
    \end{align*}
    Thus, for some constant $C>0$, we have that
    \begin{align*}
        \epsilon R_\epsilon\left(u\right) \leq C
    \end{align*}
    for sufficiently small $\epsilon $, by Lemma \ref{m_func_bdd_by_epsilon}. That, combined with the fact that $\calM(E+i\epsilon)\rightarrow\infty $ as $\epsilon\rightarrow 0$ allow us to write
    \begin{align*}
        \lim\limits_{\epsilon\to 0}\frac{\epsilon R_\epsilon\left(u\right)}{\calM\left(E+i\epsilon\right)}=0
    \end{align*}
    which proves that $g$ is indeed a solution of (\ref{ev_eq_sg}). The fact that $g$ is subordinate now follows from Lemma \ref{subordinate_solution_initial_values}.
\end{proof}

\section{Proof of Theorem \ref{main_thm}}\label{proof_section}

 The proof of Theorem \ref{main_thm} will be by contradiction. Throughout this section, We will have the setting of a general star-like graph $\calG$ with $m$ branches and a Jacobi operator  $J$ acting on it, and we will assume that  
 \begin{assumption}\label{ContradictionAssumption}
     There exists a set $A\in\text{Borel}\left(\mathbb{R}\right)$ which satisfies
 \begin{enumerate}
     \item $\mu_{\text{sc}}\left(A\right)>0$.
     \item \label{DimAssumption}$N\left(E\right)=m$ for $\mu$-almost every $E\in A$.
     \item $\text{Leb}\left(A\right)=0$.
     \item \label{NoPP} $A\cap \sigma_{pp}(J)=\emptyset$
 \end{enumerate}
 \end{assumption}
 and arrive at a contradiction.

 Recall that the space of subordinate solutions for $J$ and $\calG$, which correspond with $E\in\mathbb{R}$ is a vector space and is denoted by $\calS(E)$. Then we have the following
 \begin{lemma}\label{isomorphism_lemma}
     For $\mu_{sc}$-almost every $E\in A$, the mapping $\lambda(E):\calS(E)\rightarrow \bbR^m$ given by 
     \begin{align*}
         \lambda(E)[u]=(\lambda_i[u])_{i=1}^m
     \end{align*}
     where $\lambda_i[u]$ are the unique scalars given by Corollary \ref{UniqueScalar}, is an isomorphism.
 \end{lemma}
 \begin{proof}
     By the Assumption \ref{ContradictionAssumption}, Property \ref{DimAssumption}, we have that $N(E)=m$. Combining this with (\ref{dimension_sub_sol_multiplicity}), we get $\dim \calS(E)\geq m$, so it is enough to show that the $\ker (\lambda(E))=\{0\}$.  Let $f \in \ker (\lambda(E))$. This means that $\supp f\subset \calK$ (as it has no support on all $L_i$). On the other hand, $f\in \calS(E)$ so $f$ is a subordinate solution for energy $E$, and in particular, a solution of (\ref{ev_eq_sg}). Thus, $f$ is either an eigenfunction of $J$ or identically equal to $0$. By Property \ref{NoPP} in our Assumption \ref{ContradictionAssumption}, we conclude that $f\equiv 0$, as claimed.
 \end{proof}
It is not hard to see that given $1\leq i,j\leq m$, the scalar $\lambda_{ij}\left(E\right)=\lambda_i\left[f_j^E\right]$ can be found in a measurable way. Indeed, this is simply a matter of taking inner products of $f_j^E$ with $\delta$ functions on $\calK$, which is measurable by the measurability of the Fourier transform. Therefore, the function $E\to\Lambda\left(E\right)$, where $\Lambda\left(E\right)$ is given by $\Lambda\left(E\right)_{ij}=\lambda_{ij}\left(E\right)$ is measurable. Combining Lemma \ref{isomorphism_lemma} along with the fact that $f_1^E,\ldots,f_{N\left(E\right)}^E$ is a basis for $\Ran P\left(E\right)$, we obtain that $\Lambda$ is invertible. Clearly, the inverse $\Lambda\left(E\right)^{-1}$ is also measurable.
\begin{lemma}\label{vanishing_lemma}
    Fix $1\leq i\leq m$ and $E\in A$ and denote $x=\alpha\Lambda\left(E\right)^{-1}e_i$, where $\alpha\in\mathbb{C}$ is some non-zero scalar, and $e_i$ is the $i$th element of the standard basis.
    Let $g\in\Ran P\left(E\right)$ be given by $g=\sum\limits_{j=1}^mx_jf_j^E$.
    Then $g$ is a subordinate solution that vanishes on every half-line except the $i$-th one.
\end{lemma}
\begin{proof}
    Note that by definition we have that for every $1\leq k\leq m$, $\lambda_k\left(E\right)\left[g\right]=0$ for every $k\neq i$ and $\lambda_i\left(E\right)\left[g\right]=\alpha$. This immediately implies the result.
\end{proof}

 Finally, we get to the main conclusion, which will then allow us to get a contradiction:
 \begin{lemma}\label{wi_lemma}
     Under Assumption \ref{ContradictionAssumption}, for every $1\leq i\leq m$ there exists a subspace $\calW_i \subset \ell^2(\calG)$ with the following properties:
     \begin{enumerate}
         \item $\dim \calW_i =\infty $
         \item $\calW_i$ is $J$-invariant, namely  $J\calW_i\subset \calW_i$.
         \item For every $\psi\in \calW_i$, $\supp \psi \subset \calK\cup L_i\setminus \{\varphi_j(0)\}_{j\neq i}$.
     \end{enumerate}
 \end{lemma}
 \begin{proof}
    Let $U$ be the unitary transformation given by Theorem \ref{spect_thm_direct_int}. We define, for any $1\leq i\leq m$
    \begin{align*}
        \calW_i =\{U^{-1}\int\limits_{A}^{\oplus}h(E)\cdot\Lambda^{-1}(E)e_i\, d\mu(E)\mid h:\bbR\rightarrow\bbR \text{ is a measurable function}\}
    \end{align*}
    Since we may choose any measurable function, and in addition, the support of $\mu$ in $A$ is infinite (since it is not pure point there), we have that $\dim \calW_i =\infty $. Furthermore, by Lemma \ref{vanishing_lemma} and Lemma \ref{gen_ef_vanishing_lemma}, every  $\psi\in \calW_i$ vanishes on each $L_j$, and thus outside $\calK\cup L_i\setminus \{\varphi_j(0)\}_{j\neq i}$. Finally, we have that 
    \begin{align*}
        J\psi=U^{-1}\int\limits_{A}^{\oplus}E\cdot h(E)\cdot\Lambda^{-1}(E)e_i\, d\mu(E)
    \end{align*}
    and so by definition $J\psi\in\calW_i$, as required.
 \end{proof}
\subsection{Getting the contradiction
 }
Let us now show that the existence of a subspace $\calW_i$ given by Lemma \ref{wi_lemma} leads to a contradiction.\par
In this subsection, we will restrict ourselves to $\calG^1,$ a star-like graph with a single branch, that is $m=1$. We will consider a Jacobi operator on $J_1$ on $\calG^1$, and we will have the associated isomorphism $\varphi_1$.  \par
We start with the following observation.
\begin{claim}\label{LinearComb}
    For every $v,w \in \calV$ and $n\in \bbN$, we denote by $\calP(v,w,n)$ the collection of all paths from $v$ to $w$ of length $n$.
    Let $\psi \in \ell^2(\calG)$. Then we have, for any $n\in \bbN$ and $v\in \calV$,
    \begin{align*}
        J^n\psi(v)=\sum_{w\in B_{n-1}(v)}\alpha_w\psi(w)+\sum_{w\in \partial B_n(v)}\beta_{v,w,n} \psi(w)
    \end{align*}
    where for every $n\in\mathbb{N}$, 
    \begin{align*}
        &\beta_{v,w,n}=\sum_{(x_0,\dots, x_{n})\in \calP_{v,w,n}}\prod_{i=0}^{n-1}a_{x_i,x_{i+1}}
    \end{align*}
    and $\alpha_w $ are some constants.
\end{claim}
\begin{proof}
    This comes from simple induction. For $n=1$, we have that 
    \begin{align*}
        J\psi(v)=\sum_{w\in \partial B_1\left(v\right)} a_{v,w} \psi(w)+b_v\psi(v)
    \end{align*}
    Since the space of paths of size $1$ is exactly the neighbors, we get that $\beta_{v,w,1}=a_{v,w}$. Next, we assume the assertion is true for every vertex and every path of length $n-1$, and we can write, by assumption, for every $u\in \calV$
    \begin{align*}
        J^n\psi(u)=\sum_{w\in B_{n-1}(u)}\tilde{\alpha}_w\psi(w)+\sum_{w\in \partial B_n(u)}\beta_{w,u,n} \psi(w)
    \end{align*}
    And so we have that
    \begin{align*}
        &J^n\psi(v)=J[J^{n-1}\psi](v)=\sum_{w\in N(v)} a_{v,w}[J^{n-1}\psi](w)+b_v[J^{n-1}\psi](v)\\
        &=\sum_{w\in N(v)} a_{v,w}[\sum_{u\in B_{n-2}(w)}\tilde{\alpha}_u\psi(u)+\sum_{u\in \partial B_{n-1}(w)}\beta_{u,w,n-1} \psi(u)]\\
        &+b_v(\sum_{w\in B_{n-2}(v)}\tilde{\alpha}_w\psi(w)+\sum_{w\in \partial B_{n-1}(v)}\beta_{w,v,n-1} \psi(w))
    \end{align*}
    We note that the paths of length $n$ come from $u\in \partial B_{n-1}(w),w\sim v$, so in particular we have that 
    \begin{align*}
        \beta_{u,v,n}=\sum_{w\sim v} a_{v,w}\beta_{u,w,n-1}
    \end{align*}
    as needed.
\end{proof}
We also have the following claim:
   \begin{claim}\label{Induction}
        Let $\psi\in\ell^2\left(\calG^1\right)$ and $v_0\in\calK$. Suppose that for every $k\in\mathbb{N}\cup\left\{0\right\}$, $J_1^k\psi(v_0)=0$. Then for every $n\in\bbN$ we have that $ \psi(\varphi_1(n))$ is uniquely determined by $\{\psi(v)\}_{v\in \calK}$.
    \end{claim}
    \begin{proof}
        We will start by showing it this for $n=1$. We note that we have that $\psi(v_0)=0$, as for our assumption, and we have $J_1^\ell\psi(v_{0})=0$ for any $\ell\in \bbN$. \par
        Let $d\in \bbN$ be the smallest natural number such that $\varphi_1(1)\in \partial B_d(v_0)$. Note that $\forall j\geq 2, \varphi_1(j)\not \in \partial B_d(v_{0})$, since the only way to get to $\varphi_1(j)$ is through $\varphi_1(1)$, and $d$ is minimal.\par
        By Claim \ref{LinearComb},  we have that for every $\ell\in \bbN$
        \begin{align*}
            J_1^\ell\psi(v_{0})=\sum_{w\in  B_{n-1}(v_{0})}\alpha_{w}\psi(w)+\sum_{w\in \partial B_\ell(v_{0})}\beta_{w,v_0,\ell}\psi(w)
        \end{align*}
        where $\beta_{w,v_0,\ell}$ is exactly the sum over the weights of the edges, making all the paths from $v_{0}$ to $w$ of length $\ell$ and, in particular, is by definition non $0$. \par
        Thus, we have that 
        \begin{align*}
            0=J_1^d\psi(v_{0})=\sum_{w\in \partial B_{d-1}(v_{0})}\alpha_{w}\psi(w)+\sum_{w\in \partial B_d(v_{0})}\beta_{w,v_0,d}\psi(w)
        \end{align*}
        We note that $\varphi_1(1)\in \partial B_d(v_{0})$ but not in any previous neighboring function, and all the other terms involve only the vertices $\{w\}_{w\in \calK}$. In other words, there is a linear combination of these values that gives $0$, where the coefficient of $\psi(\varphi_1(1))$ is non-zero. So we may conclude that there are some coefficients $\{a_w\}_{w\in \calK}\in \bbR$ such that 
        \begin{align*}
            \psi(\varphi_1(1))=\sum_{w\in \calK}a_w \psi(w)
        \end{align*}
        So, we proved our claim for $n=1$. Next, we assume that the claim holds for all $j< n$, but, again, we will have that 
        \begin{align*}
            0=J_1^{d+n-1}\psi(v_{i})=\sum_{w\in B_{d+n-2}(v_{0})}\alpha_{w}\psi(w)+\sum_{w\in \partial B_{d+n-1}(v_{0})}\beta_{w,v_0,d+n-1}\psi(w)
        \end{align*}
        where $\varphi_1(n) \in \partial B_{d+n-1}(v_{i})$ and not of the previous summands, so its coefficient is non-zero. Thus we get that we have some coefficients $\{a_w\}_{w\in \calK},\{b_j\}_{j=1}^{n-1}\in \bbR$ such that 
        \begin{align*}
                \psi(\varphi_1(n))=\sum_{w\in \calK}a_w \psi(w)+\sum_{i=1}^{n-1}b_j \psi(\varphi_1(j))
        \end{align*}
        By the induction assumptions $\{\psi(\varphi_1(j))\}_{j=1}^{n-1}$ is a linear combination of $\{\psi(w)\}_{w\in \calK}$ thus completing the proof.
    \end{proof}
With this, we may prove the following claim:
\begin{claim}\label{ContraClaim}
    Let $\calG$ be a star-like graph with a single branch, that is $m=1$. Let $J^1$ be a Jacobi operator along with the associated isomorphism $\varphi_1$. Let $\calW\subset \ell^2(\calG)$ a subspace such that $J^1\calW\subset \calW$, and $\dim \calW=\infty$ then we have 
    \begin{align*}
        \forall v\in \calK, \exists \psi \in \calW, \psi(v)\neq 0
    \end{align*}
\end{claim}
\begin{proof}
    Assume, by way of contradiction, that there is some $v_0\in \calK$, such that $\forall \psi \in \calW $ we have that $\psi(v_0)=0$. \par
    By Claim \ref{Induction}, we have that $\psi$ is completely determined by the choice of $\{\psi(w)\}_{w\in \calK}$, and in particular, since this is true for arbitrary $\psi\in\calW$, this implies that $\dim \calW<\infty $ in contradiction to the assumption, thus completing the proof. 
\end{proof}
Finally, we may prove our main theorem:
\begin{proof}[Proof of Theorem \ref{main_thm}]
    Assume by way of contradiction that we have Assumption \ref{ContradictionAssumption}. Then by Lemma \ref{wi_lemma} for each $1\leq i\leq m$ we have $\calW_i$ with the following
    \begin{enumerate}
         \item $\dim \calW_i =\infty $
         \item $\calW_i$ is $J$-invariant, namely  $J\calW_i\subset \calW_i$.
         \item For every $u\in \calW_i$, $\supp u \subset \calK\cup L_i\setminus \{\varphi_j(0)\}_{j\neq i}$.
     \end{enumerate}
     In other words, we have that any $\psi \in \calW_i$ we have that $\psi(\varphi_j(0))=0$ for any $j\neq i$.\par
     But this contradicts Claim \ref{ContraClaim}, as we may restrict $\calG$ to $\calK\cup L_i$, and $\calW_i$ is an invariant subspace by that claim. So, we may conclude that  Assumption \ref{ContradictionAssumption} doesn't hold. And in particular, it means that for every $A\in\text{Borel}\left(\mathbb{R}\right)$ which satisfies
     \begin{enumerate}
         \item $\mu_{\text{sc}}\left(A\right)>0$.
         \item $\text{Leb}\left(A\right)=0$.
         \item  $A\cap \sigma_{pp}(J)=\emptyset$
     \end{enumerate}
        we must  have $N\left(E\right)<n$ for $\mu$-almost every $E\in A$, as claimed. Finally, the fact that the bound is attained is the content of Lemma  \ref{SharpnessLemma}, proven below. 
\end{proof}
\section{Examples and further discussion}\label{discuss_section}
In this section, we will consider several examples. First, we construct a star-like graph with $m$ branches and an operator such that our lower bound is satisfied. Then, we apply our theorem to some special classes of graphs that can be described as star-like, namely spherically homogeneous trees whose branching numbers are eventually 1 (precise definition below). Finally, we include some discussion about possible generalizations of Theorem \ref{main_thm}  to more general settings - of finite range operators on a star-like graph and general gluing of self-adjoint operators.
\subsection{Sharpness}\label{SharpSection}
In the following, we will provide a graph and an operator such that the multiplicity is exactly $m-1$:
\begin{lemma}\label{SharpnessLemma}
    For any $m$ there exists a star-like graph $\calG_m$ with $m$  branches, and a Jacobi operator acting on $\ell^2(\calG_m) $ such that for $\mu_{sc}$-almost  every $E\in \bbR$ we have $N(E)=m-1$. 
\end{lemma}
\begin{proof}
    Fix any half-line operator $J_0$ with purely singular continuous spectrum inside an interval, say $I=\left[-2,2\right]$. Such operators are constructed, for example, in \cite{jitomirskaya1999power}. Denote its potential by $V_0$. Let $\calG_m$ be a star graph with $m$ branches. Denote its origin by $o$. Let us denote the neighbors of $o$ by $v_1^1,\ldots,v_1^m$ and for every $j\geq 2$ and $1\leq i\leq m$, we denote by $v_j^i$ the unique vetrex in $\calV$ which satisfies $\dist\left(v_j^i,v_1^i\right)=j$. We define a Jacobi operator on $\calG_m$ by setting
    \begin{center}
        $a_{u,v}=1\iff u\sim v$,\\$\text{}$\\$b_o=0$,\\$\text{}$\\$b_{v_j^i}=V_0\left(j\right)$ for every $1\leq i\leq m$ and $j\in\mathbb{N}$.
    \end{center}
    In other words, we consider $m$ half-lines all connected at the origin to a vertex $o$ and $m$ symmetric copies of $J_0$ defined on each of these half-lines. See Figure \ref{fig:star-graph} for a visual definition of the graph and the operator.\par Fix $\zeta$, a primitive root of unity of order $m$. For every $1\leq k\leq m-1$, consider the subspace $\calH_k\subseteq\ell^2\left(\calG_m\right)$ given by
    \begin{center}
        $\calH_k=\left\{\psi\in\ell^2\left(\calG_m\right)\mid \forall1\leq i\leq m,\,\forall j\in\mathbb{N},\,\psi\left(v_j^i\right)=\zeta^{ik}\psi\left(v_j^1\right)\text{ and }\psi\left(o\right)=0\right\}$.
    \end{center}
    That is the space of functions that are the same on each branch, up to multiplication by a phase of $\zeta^{ik} $ for the $i$th branch. \par
    We first claim that this collection is pairwise orthogonal and that for every $1\leq k\leq m-1$, $\calH_k$ is $J$-invariant.\par 
    For pairwise orthogonality, Let $\psi\in \calH_k,\varphi\in \calH_\ell$ for $1\leq k<\ell\leq m-1$. Then we have that
    \begin{align*}
        \braket{\psi,\varphi}=\sum_{i=1}^m\sum_{j\in \bbN} \overline{\psi(v_j^i)}\varphi(v_j^i)+\overline{\psi(o)}\varphi(o)=\sum_{j\in \bbN}\overline{\psi(v_j^1)}\varphi(v_j^1)\sum_{i=1}^m \zeta^{i(\ell-k)}.
    \end{align*}
    Since $\zeta$ is primitive,  we have that $\sum\limits_{i=1}^m\zeta^{ik}=0$, which concludes the claim.\par
    To see that $\calH_k$ is $J$-invariant, we need to verify two conditions. First we show that for every $\psi\in\calH_k$, for every $1\leq i\leq m$ and $j\in\mathbb{N}$,  we have that $\left(J\psi\right)\left(v_j^i\right)=\zeta^{ik}\left(J\psi\right)\left(v_j^1\right)$. Indeed, for $1\neq j\in\mathbb{N}$
    \begin{align*}
        \left(J\psi\right)\left(v_j^i\right)&=\psi\left(v_{j-1}^i\right)+\psi\left(v_{j+1}^i\right)+V_0\left(j\right)\psi\left(v_j^i\right)\\
        &=\zeta^{ik}\psi\left(v_{j-1}^1\right)+\zeta^{ik}\psi\left(v_{j+1}^1\right)+\zeta^{ik}V_0\left(j\right)\psi\left(v_{j}^1\right)\\
        &=\zeta^{ik}\left(J\psi\right)\left(v_j^1\right).
    \end{align*}
    The proof for $j=1$ is almost identical, using the fact that $\psi(o)=0$.\par
    Next, we will show that for every $\psi\in\calH_k$, $\left(J\psi\right)\left(o\right)=0$. To see this, we write
    \begin{align*}
        \left(J\psi\right)\left(o\right)=\sum_{j=1}^m\psi\left(v_{1}^i\right)=\psi\left(v_{1}^1\right)\sum_{j=1}^m\zeta^{ik}=0,
    \end{align*}
    utilizing the fact that $\zeta$ is a primitive root of unity. \par 
    We have concluded that $\calH_k$ is $J$ invariant. We now claim that $J|_{\calH_k}$ is unitarily equivalent to $J_0$. Indeed, define $U:\calH_k\to\ell^2\left(\mathbb{N}\right)$ by
    \begin{center}
        $\left(U\psi\right)\left(n\right)=\frac{1}{n}\psi\left(v_n^1\right)$.
    \end{center}
    It is not hard to verify that $U$ is unitary and that it intertwines $J$ and $J_0$. To conclude, we get that $J$ is unitarily equivalent to the direct sum of $\oplus_{k=1}^{m}J^k$, where for every $1\leq k\leq m-1$, $J^k\coloneqq J|_{\calH_k}$ is unitarily equivalent to $J_0$. In that case, we have that if $N_0$ is a multiplicity function for $J_0$ and $N$ is a multiplicity function for $E$, then $N\geq\left(m-1\right)N_0$ $\mu$-almost everywhere (the proof of this fact can be found for example in \cite{ovalle2025schr}). Finally, since $J_0$ is a half-line operator, then its multiplicity function can be taken to be constant $1$ on its spectrum. Thus, for $\mu$-almost every $E\in\left[-2,2\right]$, $N\left(E\right)\geq m-1$. together with Theorem \ref{main_thm}, this implies that $N\left(E\right)= m-1$, as required. 
\end{proof}
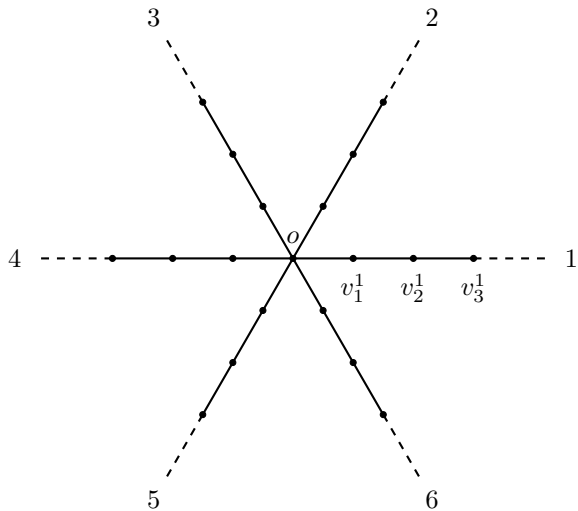
\begin{figure}[h!]
\centering
\begin{tikzpicture}[scale=2, every node/.style={font=\small}]

\def\n{6}
\def\rone{0.4}    
\def\rtwo{0.8}    
\def\rthree{1.2}  
\def\rdash{1.7}   
\def\rlabel{1.85} 

\foreach \i in {1,...,6} {
    \pgfmathsetmacro{\angle}{(\i - 1) * 360/\n}
    \draw[thick] (0,0) -- (\angle:\rone);
    \draw[thick] (\angle:\rone) -- (\angle:\rtwo);
    \draw[thick] (\angle:\rtwo) -- (\angle:\rthree);
    \draw[dashed, thick] (\angle:\rthree) -- (\angle:\rdash);
    
    \filldraw (\angle:\rone) circle (0.02);
    \filldraw (\angle:\rtwo) circle (0.02);
    \filldraw (\angle:\rthree) circle (0.02);
    
    \node at (\angle:\rlabel) {\(\i\)};
}

\filldraw (0,0) circle (0.02);
\node[above=2pt] at (0,0) {$o$};

\node[below=2pt] at (\rone,0) {$v_1^1$};
\node[below=2pt] at (\rtwo,0) {$v_2^1$};
\node[below=2pt] at (\rthree,0) {$v_3^1$};

\end{tikzpicture}
\caption{A star graph with 6 branches. Here, the vertices of $\calG_1$ are $\left\{v_1,v_2,v_3\ldots\right\}$ and the potential $b$ is given by $b\left(v_i\right)=V_0\left(i\right)$. In addition, $b\left(o\right)=0$.}
\label{fig:star-graph}
\end{figure}
\subsection{Spherically homogeneous trees and multiplicity of the spectrum}\label{branching_section}
We will now consider Jacobi operators on trees. In that case, denote the root by $o$. For every $n\in\mathbb{N}$, we will denote $B_n=B_n\left(o\right)$ and $S_n=\partial B_n (o)$. Let $\left(b_n\right)_{n\in\mathbb{N}}$ be a sequence of natural numbers. Following \cite{breuer2007singular}, we say that a tree $\calT$ is spherically homogeneous with branching numbers $\left(b_n\right)_{n\in\mathbb{N}}$ if for every $n\in\mathbb{N}$, for every $v\in S_{n-1}$,  we have $\left|\left\{u\in S_n:u\sim v\right\}\right|=b_n$. In other words, the number of neighbors of $u$ at the $n$ level is $b_n$.  We refer the reader to Figure \ref{sg_fig}, Item (c) for an example.

A simple consequence of Theorem \ref{main_thm} is the following
\begin{proposition}\label{const_branching_prop}
    Let $J$ be a Jacobi operator on a spherically homogeneous tree with branching numbers $\left(b_n\right)_{n\in\mathbb{N}}$. Suppose that there exists $N\in\mathbb{N}$ such that for every $n>N$, $b_n=1$. Let $k=\left|S_{N+1}\right|$ and let $N_J$ be the multiplicity function of $J$. Then for $\mu_{sc}$-almost every $E\in\mathbb{R}$, $N_J\left(E\right)\leq k-1$.
\end{proposition}
\begin{proof}
    It is not hard to see that every such spherically homogeneous tree is a star-like graph with $k$ branches. The result now immediately follows from Theorem \ref{main_thm}.
\end{proof}
A certain notion of dimension was defined in \cite{breuer2007singular} for such trees.
\begin{definition}
    Let $\calT$ be a spherically homogeneous tree. The dimension of $\calT$ is given by
    \begin{center}
        $\dim_o\left(\calT\right)=\underset{n\to\infty}{\limsup}\,\frac{\log\left|B_n\right|}{\log n}$.
    \end{center}
\end{definition}
Note that if $\calT$ is spherically homogeneous, which satisfies the assumptions of Proposition \ref{const_branching_prop}, then its dimension is precisely $1$. Indeed, note that there exists $N\in\mathbb{N}$ such that for every $n>N$, $S_n=k$, as defined in Proposition \ref{const_branching_prop}. With that in mind, we have
\begin{center}
    $\underset{n\to\infty}{\limsup}\,\frac{\log\left|B_n\right|}{\log n}=\underset{n\to\infty}{\limsup}\,\frac{\log\left(\sum\limits_{i=1}^n\left|S_i\right|\right)}{\log n}=\underset{n\to\infty}{\limsup}\,\frac{\log\left(\sum\limits_{i=1}^N\left|S_i\right|+\left(n-N\right)k\right)}{\log n}=\underset{n\to\infty}{\limsup}\,\frac{\log\left(nk\left(1+c_n\right)\right)}{\log n}\to 1$
\end{center}
where $c_n=\frac{1}{nk}\left(\sum\limits_{i=1}^n\left|S_i\right|-Nk\right)$, and the last assertion follows from the fact that $c_n\to 0$.

With that in mind, a natural generalization of Proposition \ref{const_branching_prop} will be the following
\begin{question}
    Let $\calT$ be a spherically homogeneous tree with dimension $1$ and let $J$ be a Jacobi operator acting on $\ell^2\left(\calT\right)$. Let $\mu$ be a Borel measure which is equivalent to the projection-valued measure associated with $J$. Is it true that for $\mu_{sc}$-almost every $E\in\mathbb{R}$, $N_J\left(E\right)<\infty$?
\end{question}

\subsection{Finite range operator on star-like graph}\label{FiniteRangeSection}
A natural generalization of the above setting is to extend this result to a more general class of operators acting on $\ell^2\left(\calG\right)$. So, one may consider a general star-like graph $\calG$ with $m$ branches and consider a finite range operator $H$. That is, for any function $\psi \in \ell^2(\calV)$, and for every $v\in \calV$ we have that $H\psi(v)$ involves $\psi(u)$ for only finitely many $u\in \calV$.\par
We note that we used the simplicity of the singular continuous spectrum in a crucial way in the construction of the invariant subspace (Lemma \ref{wi_lemma}). Thus, a natural question will be:
\begin{question}
    Given a star-like graph $\calG$ with $m$ branches and a self-adjoint finite-range operator $H:\ell^2(\calV)\rightarrow \ell^2(\calV)$, such that each restriction to a branch $H_i$ has simple singular continuous spectrum. In other words, for each $i$, we have that $H_i$ is unitarily equivalent to $H^i$, a self-adjoint operator on $\ell^2(\bbN)$ with multiplicity $1$ in the singular continuous spectrum.  Is it true that for $\mu_{sc}$-almost every $E\in\mathbb{R}$, $N_H\left(E\right)\leq m-1$?
\end{question}
The main obstacle is the lack of subordinacy theory for such operators and its connections to the generalized eigenfunction expansion. 
\subsection{General setting - pasting self-adjoint operators}\label{PastingSection}
One natural perspective on the setting above is to consider it a "gluing," a collection of Jacobi operators defined on half lines via the compact component of the graph. This point of view leads to a natural question: Can we extend this work to a general gluing of general operators? For that, we propose the following construction:
Let $\mathcal{H}_1,\ldots,\mathcal{H}_n$ be Hilbert spaces. For $i=1,\ldots,n$, let $T_i:\mathcal{H}_i\to\mathcal{H}_i$ be a self-adjoint operator with a cyclic vector $\varphi_i$. Also, let $\calC$ be a graph with $n$ vertices, and let $A$ be a weighted adjacency matrix (with real positive weights) associated with $\calC$. The pasting of $T_1,\ldots,T_n$ w.r.t.\ the pair $\left(\calC,A\right)$ is a self-adjoint operator $T$, acting on $\mathcal{H}\coloneqq\underset{i=1,\ldots,n}{\oplus}\mathcal{H}_i$. Given $\psi\in\mathcal{H}$, let us write $\psi=\left(\psi_1,\ldots,\psi_n\right)$, where $\psi_i\in\mathcal{H}_i$. Let us also denote $v=\left(\langle\psi_1,\varphi_1\rangle,\ldots,\langle\psi_n,\varphi_n\rangle\right)$. Now, $T\psi$ is defined by
\begin{center}
    $T\left(\psi_1,\ldots,\psi_n\right)=\left(T_1\psi_1,\ldots,T_n\psi_n\right)+\left((Av)_1\cdot \varphi_1,\ldots,(Av)_n\cdot\varphi_n\right)$,
\end{center}
It is not hard to verify that $T$ is indeed a self-adjoint operator. Furthermore, the set $\left\{\varphi_1,\ldots,\varphi_n\right\}$ (considered as embedded in $\calH$) is cyclic for $T$, and so its multiplicity is bounded from above by $n$ (as the multiplicity of $T_i$ is $1$).

There seems to be some connection between this approach and the one studied in \cite{simonov2024local}. However, the language in which the setting is described is substantially different. In \cite{simonov2024local}, the operator is defined via the pasting of boundary relations along a matrix. Then, it is proved that the multiplicity of the singular spectrum is bounded from above by the number of boundary relations which are non-trivial in a certain sense.

Given our definition of the pasting of self-adjoint operators with cyclic vectors, the following is a natural question.
\begin{question}
    Suppose that $\calC$ is a connected graph, and let $k=\left|\left\{i:\sigma_{sc}\left(T_i\right)\neq\emptyset\right\}\right|$. Let $\mu$ be a spectral measure that is equivalent to the projection-valued measure associated with $T$, as defined above, and let $N_T$ be a multiplicity function for $T$. Is it true that for $\mu_{sc}$-almost every $E\in\mathbb{R}$, $N_T\left(E\right)\leq k-1$?
\end{question}
We note that naturally, this more general setting contains the question asked in subsection \ref{FiniteRangeSection}.
\bibliographystyle{amsplain}
\bibliography{bib}

\end{document}